\documentclass[a4paper,11pt]{amsart}
\usepackage{amsmath,amssymb,amsthm,graphicx,mathrsfs,url}
\usepackage[usenames,dvipsnames]{color}
\usepackage{fullpage}
\usepackage{hyperref}
\usepackage{eucal}
\usepackage{euler}
\usepackage{palatino}
\usepackage{stackrel}

\hypersetup{colorlinks=true, linkcolor=blue,citecolor=citegreen}
\definecolor{citegreen}{rgb}{0.2,0.2,0.6}

\newcommand{\beq}{\begin{equation} \begin{split}}
\newcommand{\eeq}{\end{split} \end{equation}}

\newcommand{\comm}[1]{}

\def\bm1{\mathbbm{1}}

\def\dd{{\mathsf{d}}}

\newcounter{counter_a}

\usepackage[latin1]{inputenc}
\usepackage[T1]{fontenc}

\numberwithin{figure}{section}
\numberwithin{equation}{section}
\theoremstyle{plain}
\newtheorem*{thm*}{Theorem}
\newtheorem{thm}{Theorem}[section]

\theoremstyle{remark}
\newtheorem{remark}[thm]{Remark}
\theoremstyle{plain}

\newcommand{\beu}{\begin{equation*}}
\newcommand{\eeu}{\end{equation*}}
\newcommand{\besu}{\begin{equation*}
\begin{aligned}}
\newcommand{\eesu}{\end{aligned}
\end{equation*}}
\newcommand{\bes}{\begin{equation}
\begin{aligned}}
\newcommand{\ees}{\end{aligned}
\end{equation}}

\newcommand\cH{\mathcal H}

\newcommand\void[1]{}

   \def\dN{{\mathbb N}}

   \def\cH{{\mathcal H}}

\newcommand{\lm}{\lambda}
\newcommand{\ie}{\emph{i.e.}}
\newcommand{\eg}{\emph{e.g.}}
\newcommand{\cf}{\emph{cf.}}

\newcommand{\Real}{\mathbb{R}}
\newcommand{\Com}{\mathbb{C}}

\newcommand{\supp}{\mathop{\mathrm{supp}}\nolimits}

\newcommand{\Dom}{\mathsf{D}}

\newcommand{\sii}{L^2}

\newcommand{\der}{\mathrm{d}}
\newtheorem{Theorem}{Theorem}

\newtheorem{Proposition}{Proposition}

\newtheorem{Conjecture}{Conjecture}
\theoremstyle{definition}

\usepackage[normalem]{ulem}
\definecolor{DarkGreen}{rgb}{0,0.5,0.1} 

\newcommand\soutD{\bgroup\markoverwith
{\textcolor{DarkGreen}{\rule[.5ex]{2pt}{1pt}}}\ULon}

\newcommand{\Hm}[1]{\leavevmode{\marginpar{\tiny%
$\hbox to 0mm{\hspace*{-0.5mm}$\leftarrow$\hss}%
\vcenter{\vrule depth 0.1mm height 0.1mm width \the\marginparwidth}%
\hbox to
0mm{\hss$\rightarrow$\hspace*{-0.5mm}}$\\\relax\raggedright #1}}}

\begin{document}
%
\title{Optimisation of the lowest Robin eigenvalue in the exterior of a compact set}
\author{David Krej\v{c}i\v{r}{\'\i}k} 
\address{Department of Mathematics, Faculty of Nuclear Sciences and 
Physical Engineering, Czech Technical University in Prague, Trojanova 13, 12000 Prague 2,
Czech Republic}
\email{david.krejcirik@fjfi.cvut.cz}

\author{Vladimir Lotoreichik}
\address{Department of Theoretical Physics, Nuclear Physics Institute, 
Czech Academy of Sciences, 25068 \v Re\v z, Czech Republic}
\email{lotoreichik@ujf.cas.cz}
\date{\small 26 September 2016}
%

%
\begin{abstract}
\noindent
We consider the problem of geometric optimisation of the lowest eigenvalue
of the Laplacian in the exterior of a compact planar set,
subject to attractive Robin boundary conditions.
Under either a constraint of fixed perimeter or area,
we show that the maximiser within the class of exteriors of convex sets
is always the exterior of a disk. We also argue why the results fail without the convexity constraint and in higher dimensions.
\end{abstract}


\maketitle

\section{Introduction}
%
The \emph{isoperimetric inequality} states that 
among all planar sets of a given perimeter,
the disk has the largest area. 
This is equivalent to the \emph{isochoric inequality} stating that
among all planar sets of a given area,
the disk has the smallest perimeter.
These two classical geometric optimisation problems
were known to ancient Greeks, 
but a first rigorous proof appeared only in the 19th century
(see~\cite{Blasjo_2005} for an overview). 

Going from geometric to spectral quantities,
the \emph{spectral isochoric inequality} states that 
among all planar membranes of a given area and with fixed edges, 
the circular membrane produces the lowest fundamental tone. 
It was conjectured by Lord Rayleigh in 1877 in his famous book
\emph{The theory of sound} \cite{Rayleigh_1877}, 
but proved only by Faber~\cite{Faber_1923} and Krahn~\cite{Krahn_1924} 
almost half a century later.
Using scaling, 
it is easily seen that this result implies 
the \emph{spectral isoperimetric inequality} as well: 
among all planar membranes of a given perimeter and with fixed edges, 
the circular membrane produces the lowest fundamental tone. 

The same spectral inequalities under the area or perimeter constraints
extend to elastically supported membranes.
To be more precise, let $\Omega \subset \Real^2$ be 
any smooth bounded open set.
Given a real number~$\alpha$, consider the spectral problem 
for the Laplacian, subject to Robin boundary conditions,
\begin{equation}\label{problem}
\left\{
\begin{aligned}
  -\Delta u &= \lambda u 
  && \mbox{in} \quad \Omega \,,
  \\
  \frac{\partial u}{\partial n}+\alpha \;\! u&=0 
  && \mbox{on} \quad \partial\Omega \,,
\end{aligned}
\right.
\end{equation}
where~$n$ is the \emph{outer} unit normal to~$\Omega$.
It is well known that~\eqref{problem} induces an infinite number 
of eigenvalues 
$
  \lambda_1^\alpha(\Omega) \le \lambda_2^\alpha(\Omega) 
  \leq \lambda_3^\alpha(\Omega) \leq \dots
$ 
diverging to infinity.
The lowest eigenvalue admits the variational characterisation
\begin{equation}\label{Rayleigh}
  \lambda_1^\alpha(\Omega)
  = \inf_{\stackrel[u\not=0]{}{u \in W^{1,2}(\Omega)}}
  \frac{\displaystyle \int_\Omega |\nabla u|^2 + \alpha \int_{\partial\Omega} |u|^2}
  {\displaystyle \, \int_{\Omega} |u|^2}
  \,,
\end{equation}
from which it is clear that $\lambda_1^\alpha(\Omega)$ is non-negative 
if, and only if, $\alpha$~is non-negative.
In this elastic regime, the spectral isoperimetric and isochoric inequalities 
for the Robin Laplacian can be respectively stated as follows
\begin{flalign}\label{Bossel}
  \fbox{$\alpha \geq 0$} &&
  \min_{|\partial\Omega| = c_1} \lambda_1^\alpha(\Omega) 
  = \lambda_1^\alpha(B_{R_1})
  \qquad \mbox{and} \qquad
  \min_{|\Omega| = c_2} \lambda_1^\alpha(\Omega) 
  = \lambda_1^\alpha(B_{R_2})
  \,.
  &&
\end{flalign}
Here~$B_{R_1}$ and~$B_{R_2}$ are disks of perimeter $|\partial B_{R_1}| = c_1$
and area $|B_{R_2}| = c_2$, respectively, 
with~$c_1$ and~$c_2$ being two arbitrary positive constants.
With an abuse of notation, we denote by~$|\Omega|$ and~$|\partial\Omega|$ 
the $2$-dimensional Lebesgue measure of~$\Omega$
and the $1$-dimensional Hausdorff measure of its boundary~$\partial\Omega$,
respectively. 
In the Dirichlet case (formally corresponding to setting $\alpha=\infty$), 
the second optimisation problem in~\eqref{Bossel} 
is just the Rayleigh-Faber-Krahn inequality mentioned above.
Both the statements in~\eqref{Bossel} are trivial 
for Neumann boundary conditions ($\alpha=0$).
If~$\alpha > 0$, the second result in~\eqref{Bossel} 
is due to Bossel~\cite{Bossel_1986} from 1986
(and Daners~\cite{Daners_2006} from 2006, also for higher dimensions),
while the first identity again follows by scaling.

Summing up, going from the ancient isoperimetric inequality
to the most recent spectral results on the Robin problem with positive 
boundary coupling parameter, the disk turns out to be always the extremal set
for the optimisation problems under the area or perimeter constraints.
Moreover, the isoperimetric and isochoric optimisation problems are closely related. 
In the last two years, however, it has been noticed that the optimisation
of the Robin eigenvalue in the case of \emph{negative}~$\alpha$ 
is quite different. Here the story goes back to 1977 
when Bareket conjectured that a \emph{reverse spectral isochoric inequality}
should hold:
\setcounter{Conjecture}{-1}
\begin{Conjecture}[Bareket's \emph{reverse spectral isochoric inequality}~\cite{Bareket_1977}]
\label{Conj}
For all negative~$\alpha$, we have
\begin{flalign}\label{Bareket}
  \fbox{$\alpha < 0$} &&
  \max_{|\Omega| = c_2} \lambda_1^\alpha(\Omega) 
  = \lambda_1^\alpha(B_{R_2})
  \,,
  &&
\end{flalign}
where the maximum is taken over all bounded open connected sets~$\Omega$ 
of a given area~$c_2 > 0$
and~$B_{R_2}$ is the disk of the same area as~$\Omega$, 
\ie\ $|B_{R_2}| = c_2$.   
\end{Conjecture}

Notice that, contrary to~\eqref{Bossel}, it is natural to maximise  
the eigenvalue if~$\alpha$ is negative. 
The conjecture was supported 
by its validity for sets close to the disk
~\cite{Bareket_1977,Ferone-Nitsch-Trombetti_2015}
and revived both in~\cite{Brock-Daners_2007} and~\cite{Daners_2013}.
However, Freitas and one of the present authors
have recently disproved the conjecture~\cite{FK7}:
While~\eqref{Bareket} holds for small negative values of~$\alpha$,
it cannot hold for all values of the boundary parameter.
In fact, the annulus provides a larger value of the lowest 
Robin eigenvalue as $\alpha \to -\infty$.
This provides the first known example where
the extremal domain for the lowest eigenvalue of 
the Robin Laplacian is not a disk.
On the other hand, in the most recent publication~\cite{AFK},
it was shown that the \emph{reverse spectral isoperimetric inequality}
does hold:
\setcounter{Theorem}{-1}
\begin{Theorem}[\emph{Reverse spectral isoperimetric inequality}~\cite{AFK}]
For all negative~$\alpha$, we have
\begin{flalign}\label{isoperimeter}
  \fbox{$\alpha < 0$} &&
  \max_{|\partial\Omega| = c_1} \lambda_1^\alpha(\Omega) 
  = \lambda_1^\alpha(B_{R_1})
  \,,
  &&
\end{flalign}
where the maximum is taken over all smooth bounded open connected sets~$\Omega$ 
of a given perimeter~$c_1 > 0$
and~$B_{R_1}$ is the disk of the same perimeter as~$\Omega$, 
\ie\ $|\partial B_{R_1}| =c_1$.   
\end{Theorem}

Summing up, when~$\alpha$ is negative, the disk is still the optimiser
of the reverse spectral isoperimetric problem,
while it stops to be the optimiser of the isochoric problem 
for large negative values of~$\alpha$.
Whether the optimiser becomes the annulus for the larger negative values of~$\alpha$
and whether Conjecture~\ref{Conj} holds under some geometric restrictions on~$\Omega$
represent just a few hot open problems in the recent study
(see~\cite[Sec.~5.3]{AFK} for conjectures supported by numerical experiments).    

In this paper, we show that both the reverse spectral 
isochoric and isoperimetric inequalities hold in the dual setting
of the Robin problem in the \emph{exterior of a convex set}.
To this purpose, for any open set $\Omega \subset \Real^2$,
we define $\Omega^\mathrm{ext} := \Real^2 \setminus \overline{\Omega}$.
\begin{Theorem}\label{Thm}
For all negative~$\alpha$, we have
\begin{flalign}\label{results}
  \fbox{$\alpha < 0$} &&
  \max_{\stackrel[\Omega~\text{\rm convex}]{}{|\partial\Omega| = c_1}}
  \lambda_1^\alpha(\Omega^\mathrm{ext}) 
  = \lambda_1^\alpha(B_{R_1}^\mathrm{ext})
  \qquad \mbox{and} \qquad
  \max_{\stackrel[\Omega~\text{\rm convex}]{}{|\Omega| = c_2}}
  \lambda_1^\alpha(\Omega^\mathrm{ext}) 
  = \lambda_1^\alpha(B_{R_2}^\mathrm{ext})
  \,,
  &&
\end{flalign}
where the maxima are taken over all convex smooth bounded open sets~$\Omega$ 
of a given perimeter~$c_1 > 0$ or area~$c_2 > 0$, respectively,
and~$B_{R_1}$ and~$B_{R_2}$ are disks of perimeter $|\partial B_{R_1}| = c_1$
and area $|B_{R_2}| = c_2$.
\end{Theorem}

It is important to mention that $\lambda_1^\alpha(\Omega^\mathrm{ext})$ 
when defined by~\eqref{Rayleigh} indeed represents a (negative) \emph{discrete} eigenvalue
of a self-adjoint realisation of the Robin Laplacian in $\sii(\Omega^\mathrm{ext})$. 
It is not obvious because there is also the essential spectrum $[0,\infty)$,
but it can be shown by using the criticality of the Laplacian in~$\Real^2$
and the fact that~$\alpha$ is negative (\cf~Section~\ref{Sec.pre}).
In fact, $\lambda_1^\alpha(\Omega^\mathrm{ext})$ equals zero
(the lowest point in the essential spectrum) for any domain~$\Omega$
if~$\alpha$ is non-negative,
so the optimisation problems~\eqref{results} are trivial in this case. 
At the same time, because of the existence of the Hardy inequality in higher dimensions,
$\lambda_1^\alpha(\Omega^\mathrm{ext})$ is also zero 
for all small (negative) values of~$\alpha$ 
if the dimension is greater than or equal to three,
so the identities~\eqref{results} become trivial in this regime, too.
This is the main reason why 
we mostly (but not exclusively) restrict to planar domains in this paper.
As a matter of fact, in Section~\ref{Sec.high} we argue that 
an analogue of Theorem~\ref{Thm} can not hold in space
dimensions greater than or equal to three.

We point out that
the disks are the only optimisers in Theorem~\ref{Thm}
(see Section~\ref{Sec.uniqueness} for the respective argument).
It is worth mentioning that the identities~\eqref{results} 
are no longer valid if the condition of
connectedness of $\Omega$ is dropped 
(see Section~\ref{Sec.counterexample} for a counterexample). 
However, it is still unclear at the moment whether the convexity of $\Omega$ in~\eqref{results} can be replaced by a weaker assumption. 

The organisation of this paper is as follows.
In Section~\ref{Sec.pre} we provide an operator-theoretic framework for
the eigenvalue problem of Robin type~\eqref{problem}
and establish basic spectral properties in the exterior of a compact set.
Section~\ref{Sec.disk} is devoted to more specific results 
on the lowest eigenvalue in the case of the compact set being a disk.
Theorem~\ref{Thm} is proved in Section~\ref{Sec.proof}:
The isoperimetric part of the theorem follows quite straightforwardly
by using test functions with level lines parallel to the boundary~$\partial\Omega$,
while the isochoric result is established with help of scaling 
and a monotonicity result of Section~\ref{Sec.disk}.
The method of parallel coordinates was employed 
also in~\cite{FK7} to establish Conjecture~\ref{Conj}
for small values of~$\alpha$,
however, the reader will notice that
the present implementation of the technique is quite different 
and in principal does not restrict to planar sets (\cf~Section~\ref{Sec.three}).
The paper is concluded by Section~\ref{Sec.end} with more comments
on our results and methods.

\section{The spectral problem in the exterior of a compact set}\label{Sec.pre}
%
Throughout this section, $\Omega$~is an arbitrary bounded open set in~$\Real^2$,
not necessarily connected or convex.
However, a standing assumption is that the exterior~$\Omega^\mathrm{ext}$ is connected. 
Occasionally,
we adopt the shorthand notation $\Sigma := \partial\Omega$.
For simplicity, we assume that~$\Omega$ is 
smooth (\ie~$C^\infty$-smooth),
but less regularity is evidently needed for the majority of the results to hold.
At the same time, $\alpha$~is an arbitrary real number, 
not necessarily negative (unless otherwise stated).

We are interested in the eigenvalue problem
\begin{equation}\label{problem.ext}
\left\{
\begin{aligned}
  -\Delta u &= \lambda u 
  && \mbox{in} \quad \Omega^\mathrm{ext} \,,
  \\
  -\frac{\partial u}{\partial n}+\alpha \;\! u&=0 
  && \mbox{on} \quad \partial\Omega^\mathrm{ext} \,.
\end{aligned}
\right.
\end{equation}
We recall that~$n$ is the \emph{outer} unit normal to~$\Omega$,
that is why we have the flip of sign with respect to~\eqref{problem}.
As usual, we understand~\eqref{problem.ext} as the spectral problem for
the self-adjoint operator 
$-\Delta_\alpha^{\Omega^\mathrm{ext}}$ in $\sii(\Omega^\mathrm{ext})$
associated with the closed quadratic form
\begin{equation}\label{form}
  Q_\alpha^{\Omega^\mathrm{ext}}[u] := 
  \|\nabla u\|_{\sii(\Omega^\mathrm{ext})}^2  +\alpha\;\!\|u\|_{\sii(\Sigma)}^2
  \,, \qquad 
  \Dom(Q_\alpha^{\Omega^\mathrm{ext}}) := W^{1,2}(\Omega^\mathrm{ext})
  \,.
\end{equation}
The boundary term is understood in the sense of traces
$W^{1,2}(\Omega^\mathrm{ext}) \hookrightarrow \sii(\Sigma)$
and represents a relatively bounded perturbation of 
the Neumann form~$Q_0^{\Omega^\mathrm{ext}}$
with the relative bound equal to zero.
Since~$\Omega$ is smooth, 
the operator domain of $-\Delta_\alpha^{\Omega^\mathrm{ext}}$ 
consists of functions $u \in W^{2,2}(\Omega^\mathrm{ext})$
which satisfy the Robin boundary conditions 
of~\eqref{problem.ext} in the sense of traces
and the operator acts as the distributional Laplacian 
(\cf~\cite[Thm.~3.5]{BLLR15} for the $W^{2,2}$-regularity).
We call $-\Delta_\alpha^{\Omega^\mathrm{ext}}$
the \emph{Robin Laplacian} in~$\Omega^\mathrm{ext}$.

Since~$\Omega$ is bounded, 
the embedding $W^{1,2}(\Omega^\mathrm{ext}) \hookrightarrow \sii(\Omega^\mathrm{ext})$
is \emph{not} compact.
In fact, the Robin Laplacian possesses a non-empty essential spectrum
which equals $[0,\infty)$.
This property is expected because the (essential) spectrum of the Laplacian
in the whole space~$\Real^2$ (\ie~$\Omega=\varnothing$) equals $[0,\infty)$ 
and removing~$\Omega$ can be understood as a compact perturbation. 
In order to keep the paper self-contained, 
we provide a proof of this claim which relies
on an explicit construction of singular sequences 
and a Neumann bracketing argument.
\begin{Proposition}\label{Prop.ess}
We have
$
  \sigma_\mathrm{ess}(-\Delta_\alpha^{\Omega^\mathrm{ext}})
  = [0,\infty)
  \,.
$
\end{Proposition}
\begin{proof}
First, we show the inclusion  
$
  \sigma_\mathrm{ess}(-\Delta_\alpha^{\Omega^\mathrm{ext}})
  \supset [0,\infty)
$
by constructing a suitable singular sequence.
For any positive integer~$n$, let us set
$
  u_n(x) := \varphi_n(x) \, e^{ik \cdot x}
$
with an arbitrary vector $k \in \Real^2$ and
$
  \varphi_n(x) := n^{-1} \varphi((x-n x_0)/n)
$,
where~$\varphi$ is a function from $C_0^\infty(\Real^2)$
normalised to~$1$ in $\sii(\Real^2)$ and $x_0:=(1,0)$. 
The prefactor in the definition of~$\varphi_n$ is chosen in such a way
that~$\varphi_n$ is normalised to~$1$ in $\sii(\Real^2)$ for each~$n$.
In fact, we have
\begin{equation}\label{varphi.singular}
  \|\varphi_n\|_{\sii(\Real^2)}=1 
  \,, \quad
  \|\nabla\varphi_n\|_{\sii(\Real^2)}= n^{-1} \|\nabla\varphi\|_{\sii(\Real^2)} 
  \,, \quad
  \|\Delta\varphi_n\|_{\sii(\Real^2)}= n^{-2} \|\Delta\varphi\|_{\sii(\Real^2)}
  \,.
\end{equation}
At the same time, the support of~$\varphi_n$ leaves 
any bounded set for all sufficiently large~$n$. 
Consequently, for all sufficiently large~$n$, we have
$
  u_n \in C_0^\infty(\Omega^\mathrm{ext}) 
  \subset \Dom(-\Delta_\alpha^{\Omega^\mathrm{ext}})
$
and $\|u_n\|_{\sii(\Omega^\mathrm{ext})} = 1$.
A direct computation yields
\begin{equation*}
  \left|-\Delta_\alpha^{\Omega^\mathrm{ext}}u_n-|k|^2 u_n\right|
  = \left|(-\Delta\varphi_n+2i k\cdot\nabla\varphi_n)
  \,e^{ik \cdot x}\right|
  \leq |\Delta\varphi_n| + 2\,|k| |\nabla\varphi_n|
  \,.
\end{equation*}
Using~\eqref{varphi.singular}, we therefore have
$$
  \big\|
  -\Delta_\alpha^{\Omega^\mathrm{ext}} u_n-|k|^2 u_n
  \big\|_{\sii(\Omega^\mathrm{ext})}^2
  \leq 2 \|\Delta\varphi_n\|_{\sii(\Real^2)}^2 
  + 8\,|k|^2 \|\nabla\varphi_n\|_{\sii(\Real^2)}^2
  \xrightarrow[n\to\infty]{} 0
  \,.
$$
Since~$k$ is arbitrary, we conclude that 
$[0,\infty) \subset \sigma(-\Delta_\alpha^{\Omega^\mathrm{ext}})$
by \cite[Thm.~7.22]{Weidmann}.
It is clear that $[0,\infty)$ 
actually belongs to the \emph{essential} spectrum,
because the interval has no isolated points.

Second, to show the opposite inclusion 
$
  \sigma_\mathrm{ess}(-\Delta_\alpha^{\Omega^\mathrm{ext}})
  \subset [0,\infty)
$,
we use a Neumann bracketing argument. 
Let~$H_n$ be the operator that acts as $-\Delta_\alpha^{\Omega^\mathrm{ext}}$
but satisfies an extra Neumann condition on the circle 
$\mathcal{C}_n := \{x\in\Real^2:|x|=n\}$
of radius~$n>0$.
More specifically, $H_n$~is the operator associated with the form 
\begin{equation*}
  h_n[u] := \|\nabla u\|_{\sii(\Omega^\mathrm{ext})}^2 
  +\alpha\;\!\|u\|_{\sii(\Sigma)}^2
  \,, \qquad 
  \Dom(h_n) 
  := W^{1,2}(\Omega^\mathrm{ext} \setminus \mathcal{C}_n) 
  \,.
\end{equation*}
Because of the domain inclusion 
$\Dom(h_n) \supset \Dom(Q_\alpha^{\Omega^\mathrm{ext}})$,
we have $-\Delta_\alpha^{\Omega^\mathrm{ext}} \geq H_n$
and, by the min-max principle,
$
  \inf\sigma_\mathrm{ess}(-\Delta_\alpha^{\Omega^\mathrm{ext}})
  \geq \inf\sigma_\mathrm{ess}(H_n)
$
for all~$n$.
Assuming that~$n$ is sufficiently large
so that~$\Omega$ is contained in the disk $B_n := \{x\in\Real^2:|x|<n\}$, 
$H_n$~decouples into an orthogonal
sum of two operators, $H_n = H_n^{(1)} \oplus H_n^{(2)}$
with respect to the decomposition $L^2(\Omega^{\rm ext}) = 
L^2(\Omega^{\rm ext}\cap B_n) \oplus L^2(\Omega^{\rm ext}\setminus\overline{B_n})$.
Here $H_n^{(1)}$ and $H_n^{(2)}$ are respectively 
the operators in $\sii(\Omega^\mathrm{ext} \cap B_n)$
and $\sii(\Omega^\mathrm{ext} \setminus \overline{B_n})$
associated with the forms
\begin{align*}
  h_n^{(1)}[u] 
  &:= \|\nabla u\|_{\sii(\Omega^\mathrm{ext} \cap B_n)}^2 
  +\alpha\;\!\|u\|_{\sii(\Sigma)}^2
  \,, \qquad 
  &
  \Dom(h_n^{(1)}) 
  &:= W^{1,2}(\Omega^\mathrm{ext} \cap B_n) 
  \,,
  \\
  h_n^{(2)}[u] 
  &:= \|\nabla u\|_{\sii(\Omega^\mathrm{ext} \setminus \overline{B_n})}^2 
  \,, \qquad 
  &
  \Dom(h_n^{(2)}) 
  &:= W^{1,2}(\Omega^\mathrm{ext} \setminus \overline{B_n}) 
  \,.
\end{align*}
Since $\Omega^\mathrm{ext} \cap B_n$ is a smooth bounded open set,
the spectrum of $H_n^{(1)}$ is purely discrete.
Consequently,
$
  \inf\sigma_\mathrm{ess}(-\Delta_\alpha^{\Omega^\mathrm{ext}})
  \geq \inf\sigma_\mathrm{ess}(H_n^{(2)})
  \geq \inf\sigma(H_n^{(2)})
  \geq 0
$,
where the last inequality follows by the fact that~$H_n^{(2)}$ is non-negative. 
\end{proof} 

Despite of the presence of essential spectrum, 
it still makes sense to define the lowest point in the spectrum  
of $-\Delta_\alpha^{\Omega^\mathrm{ext}}$ 
by the variational formula (\cf~\eqref{Rayleigh})
\begin{equation}\label{Rayleigh.ext}
  \lambda_1^\alpha(\Omega^\mathrm{ext})
  := \inf_{\stackrel[u\not=0]{}{u \in W^{1,2}(\Omega^\mathrm{ext})}}
  \frac{\displaystyle Q_\alpha^{\Omega^\mathrm{ext}}[u]}
  {\displaystyle \, \|u\|_{\sii(\Omega^\mathrm{ext})}^2}
  \,.
\end{equation}
However, it is not evident that it represents a discrete eigenvalue 
of $-\Delta_\alpha^{\Omega^\mathrm{ext}}$.
Obviously, it is not the case if~$\alpha$ is non-negative,
in which case $-\Delta_\alpha^{\Omega^\mathrm{ext}}$ is non-negative
and therefore its spectrum is purely essential.
The following result shows that the situation of negative~$\alpha$ 
is different.
\begin{Proposition}\label{Prop.disc}
If $\alpha < 0$ and~$\Omega$ is not empty, then
$
  \sigma_\mathrm{disc}(-\Delta_\alpha^{\Omega^\mathrm{ext}})
  \not= \varnothing
  \,.
$
More specifically, $\lambda_1^\alpha(\Omega^\mathrm{ext})$ 
is a negative discrete eigenvalue.
\end{Proposition}
\begin{proof}
By Proposition~\ref{Prop.ess} and~\eqref{Rayleigh.ext},
it is enough to find a test function $u \in W^{1,2}(\Omega^\mathrm{ext})$ 
such that $Q_\alpha^{\Omega^\mathrm{ext}}[u]$ is negative.
For any positive number~$n$, 
we introduce a function $u_n \colon \Real^2 \to [0,1]$
by setting $u_n(x) := \varphi_n(|x|)$ with
$$
  \varphi_n(r) :=
  \begin{cases}
    1
    & \mbox{if} \quad r < n \,,
    \\
    \displaystyle
    \frac{\log n^2 - \log r}{\log n^2 - \log n}
    & \mbox{if} \quad  n < r < n^2 ,
    \\
    0
    & \mbox{otherwise} \,.
  \end{cases}
$$
It is not difficult to check that the restriction of~$u_n$ to~$\Omega^\mathrm{ext}$
(that we shall denote by the same symbol)
belongs to $W^{1,2}(\Omega^\mathrm{ext})$ for every~$n$.
By employing polar coordinates, we have 
$$
  \|\nabla u_n\|_{\sii(\Omega^\mathrm{ext})}^2 
  \leq \|\nabla u_n\|_{\sii(\Real^2)}^2
  = 2\pi \int_0^\infty 
  |\varphi_n'(r)|^2 \, r \, \der r 
  =
  2\pi \int_n^{n^2} 
  \frac{1}{(\log n)^2\,r} \, \der r
  =  \frac{2\pi}{\log n}
  \xrightarrow[n\to\infty]{} 0
  \,.
$$
On the other hand, 
$
  \|u_n\|_{\sii(\Sigma)}^2 
  = |\Sigma| 
  > 0
$
for all sufficiently large~$n$.
Since~$\alpha$ is assumed to be negative, 
it follows that $Q_\alpha^{\Omega^\mathrm{ext}}[u_n]$
can be made negative for~$n$ large enough.
\end{proof} 

Summing up, if~$\alpha$ is negative, the essential spectrum of
$-\Delta_\alpha^{\Omega^\mathrm{ext}}$ equals the interval $[0,\infty)$
and there is at least one discrete eigenvalue below~$0$.
In particular, the lowest point $\lambda_1^\alpha(\Omega^\mathrm{ext})$ 
in the spectrum is always a negative discrete eigenvalue. 
By standard methods (see, \eg, \cite[Thm.~8.38]{Gilbarg-Trudinger}),
it is possible to show that~$\lambda_1^\alpha(\Omega^\mathrm{ext})$ is simple
and that the corresponding eigenfunction~$u_1^\alpha$ can be chosen 
to be positive in~$\Omega^\mathrm{ext}$
(recall that we always assume that~$\Omega^\mathrm{ext}$ is connected).

It is straightforward to verify that 
$\{Q_\alpha^{\Omega^\mathrm{ext}}\}_{\alpha\in\Real}$
is a holomorphic family of forms of type~(a) in the sense of Kato~\cite[Sec.~VII.4]{Kato}. 
In fact, recalling that the boundary term in~\eqref{form} 
is relatively bounded with respect to the Neumann form~$Q_0^{\Omega^\mathrm{ext}}$
with the relative bound equal to zero,
one can use~\cite[Thm.~4.8]{Kato} to get the claim.
Consequently, $-\Delta_\alpha^{\Omega^\mathrm{ext}}$ forms 
a self-adjoint holomorphic family of operators of type~(B).
Because of the simplicity, it follows that 
$\alpha \mapsto \lambda_1^\alpha(\Omega^\mathrm{ext})$ 
and $\alpha \mapsto u_1^\alpha$ with $\|u_1^\alpha\| =1$ 
are real-analytic functions on~$(-\infty,0)$. 
\begin{Proposition}\label{Prop.convex}
	Let $\alpha < 0$ and $\Omega \not=\varnothing$.
	Then $\alpha \mapsto \lambda_1^\alpha(\Omega^\mathrm{ext})$
	is a strictly concave increasing function.
\end{Proposition}
\begin{proof}
	For simplicity, let us set 
	$\lambda_1^\alpha := \lambda_1^\alpha(\Omega^\mathrm{ext})$.
	The eigenvalue equation 
	$-\Delta_\alpha^{\Omega^\mathrm{ext}} u_1^\alpha = 
	\lambda_1^\alpha \, u_1^\alpha$
	means that 
	\begin{equation}\label{eq:der0}
  	\forall \varphi \in W^{1,2}(\Omega^\mathrm{ext})
	  \,, \qquad
  	Q_\alpha^{\Omega^\mathrm{ext}}(\varphi,u_1^\alpha) 
	  = \lambda_1^\alpha \, (\varphi,u_1^\alpha)_{\sii(\Omega^\mathrm{ext})}
  \,.
	\end{equation}
	Differentiating the identity~\eqref{eq:der0} with respect to~$\alpha$,
	we easily arrive at the formula
	\begin{equation}\label{eq:der1}
		(\nabla \varphi,\nabla \dot{u}^\alpha_1)_{L^2(\Omega^{\rm ext})}
		+ 
		(\varphi, u^\alpha_1)_{L^2(\Sigma)}
		+ 
		\alpha(\varphi,	\dot{u}^\alpha_1)_{L^2(\Sigma)}
		=
		\dot{\lambda}_1^\alpha(\varphi,u_1^\alpha)_{L^2(\Omega^{\rm ext})}
		+
		\lambda_1^\alpha(\varphi,\dot{u}_1^\alpha)_{L^2(\Omega^{\rm ext})}
                \,,
	\end{equation}
where the dot denotes the derivative with respect to~$\alpha$.
Notice that the differentiation below the integral signs is permitted
because $\dot{u}_1^\alpha \in W^{1,2}(\Omega^\mathrm{ext})$
by standard elliptic regularity theory. 
Moreover, differentiating the normalisation condition $\|u_1^\alpha\| =1$,
we get the orthogonality property
\begin{equation}\label{orthogonal}
  (u_1^\alpha,\dot{u}_1^\alpha)_{L^2(\Omega^\mathrm{ext})}= 0
  \,.
\end{equation}

	Now, substituting $\varphi = u^\alpha_1$ into~\eqref{eq:der1} and 
	$\varphi = \dot{u}_1^\alpha$ 
	into~\eqref{eq:der0} and taking the difference of the resulting
	equations, we get a formula for the eigenvalue derivative
	\begin{equation}\label{eq:lm1_der}
  		\dot{\lambda}_1^\alpha  
		= 
		\|u_1^\alpha\|^2_{\sii(\Sigma)} > 0.
	\end{equation}
	The above inequality is strict because otherwise~$u_1^\alpha$ would be
	an eigenfunction of the Dirichlet Laplacian on $\Omega^\mathrm{ext}$
	corresponding to a negative eigenvalue $\lambda_1^\alpha$, 
    which is a contradiction to the non-negativity of the latter operator.
	This proves that $\alpha \mapsto \lambda_1^\alpha$
	is strictly increasing.

	Next, 
	we differentiate equation~\eqref{eq:lm1_der} with respect to $\alpha$,
	\begin{equation}\label{2nd}
		  	\ddot{\lambda}_1^\alpha 
		  	= 
		  	\frac{\dd}{\dd\alpha}\left(\|u_1^\alpha\|_{\sii(\Sigma)}^2\right)
	  		 = 2(u_1^\alpha,\dot{u}_1^\alpha)_{\sii(\Sigma)}\\
	  		= 
	  		2\lambda_1^\alpha\|\dot{u}_1^\alpha\|^2_{\sii(\Omega^\mathrm{ext})}
	  		 -2Q^{\Omega^\mathrm{ext}}_\alpha[\dot{u}_1^\alpha] < 0 \,.
	\end{equation}	
	Here the last equality employs~\eqref{eq:der1} 
with the choice $\varphi = \dot{u}_1^\alpha$ and~\eqref{orthogonal}.
The inequality follows from the fact that $\lm_1^\alpha$
is the lowest eigenvalue of $-\Delta^{\Omega^\mathrm{ext}}_\alpha$.		
	The above inequality is indeed strict 
	since otherwise $\dot{u}_1^\alpha$ would be either another eigenfunction 
	of~$-\Delta_\alpha^{\Omega^\mathrm{ext}}$ corresponding to~$\lambda_1^\alpha$,
	which is impossible because of the simplicity, 
	or a constant multiple of~$u_1^\alpha$,
	which would imply $\dot{u}_1^\alpha = 0$
	due to~\eqref{orthogonal}.
In the latter case~\eqref{eq:der1} 		
	gives
	\[
	  	\forall \varphi \in C^\infty_0(\Omega^\mathrm{ext})
	  \,, \qquad
	  	(\varphi,u_1^\alpha)_{L^2(\Omega^{\rm ext})} =0 \,,
	\]
	and therefore $u_1^\alpha = 0$, which is also a contradiction.
	From~\eqref{2nd} we therefore conclude that
$\alpha \mapsto \lambda_1^\alpha$ is strictly concave.
\end{proof} 

As a consequence of~Proposition~\ref{Prop.convex}, we get
\begin{equation} 
  \lim_{\alpha \to -\infty} \lambda_1^\alpha(\Omega^\mathrm{ext})
  = -\infty 
  \,.
\end{equation}
%

\section{The lowest eigenvalue in the exterior of a disk}\label{Sec.disk}
%
In this section, we establish some properties of 
$\lambda_1^\alpha(B_R^\mathrm{ext})$,
where~$B_R$ is an open disk of radius $R>0$. 
Without loss of generality, we can assume that~$B_R$
is centred at the origin of~$\Real^2$.
We always assume that~$\alpha$ is negative.

Using the rotational symmetry, 
it is easily seen that $\lambda_1^\alpha(B_R^\mathrm{ext})=-k^2 < 0$
is the smallest solution of the ordinary differential spectral problem
\begin{equation}\label{ODE}
\left\{
\begin{aligned}  
  -r^{-1} [r \psi'(r)]' &= \lambda \psi(r) \,,
  && r \in (R,\infty) \,,
  \\
  -\psi'(R)+\alpha\,\psi(R) &= 0 \,,
  \\
  \lim_{r \to \infty} \psi(r) &= 0 \,.
\end{aligned}
\right.
\end{equation}
The general solution of the differential equation in~\eqref{ODE}
is given by
\begin{equation}\label{Bessel.torus}
  \psi(r) = 
  C_1 K_{0}(kr) + C_2 I_{0}(kr)
  \,, \qquad
  C_1, C_2 \in \Com
  \,,
\end{equation}
where $K_0, I_0$ are modified Bessel functions 
\cite[Sec.~9.6]{Abramowitz-Stegun}. 
The solution~$I_{0}(kr)$ is excluded because it diverges as $r\to\infty$,
whence $C_2=0$.
Requiring $\psi$ to satisfy the Robin boundary condition at~$R$
leads us to the implicit equation 
\begin{equation}\label{implicit}
  k K_1(k R) + \alpha K_0(k R) = 0 
\end{equation}
that~$k$ must satisfy as a function of~$\alpha$ and~$R$.   

First of all, we state the following upper and lower bounds.
\begin{Proposition}\label{Prop.bounds}
We have 
$$
  -\alpha^2
  < \lambda_1^\alpha(B_R^\mathrm{ext}) <
  -\alpha^2 - \frac{\alpha}{R}
$$
for all negative~$\alpha$.
\end{Proposition}
\begin{proof}
For simplicity, let us set $\lambda_R := \lambda_1^\alpha(B_R^\mathrm{ext})$
and recall the notation $\lambda_R =  -k^2$. 
Using~\eqref{implicit} and \cite[Eq.~74]{Segura_2011} (with $\nu = 0$), 
we have
\[
	\lambda_R =  -k^2 = 
	-\alpha^2\left(\frac{K_0(kR)}{K_1(kR)}\right)^2 > -\alpha^2.
\]
This establishes the lower bound of the proposition.
In the case $\alpha \in (-R^{-1},0)$, 
we obtain the upper bound from the elementary estimate
\[
	\lambda_R+\alpha^2+\frac{\alpha}{R} < \alpha^2+\frac{\alpha}{R} = 
	\alpha\left(\alpha+\frac{1}{R}\right) < 0
  \,,  
\]
where we have used the fact that~$\lambda_R$ is negative.
In the other case $\alpha \in (-\infty, -R^{-1}]$, 
we get by \cite[Thm.~1]{Segura_2011} (with $\nu = 1/2$) that
\[
	k^2 = 
	\alpha^2\left(\frac{K_0(kR)}{K_1(kR)}\right)^2
	>
	\frac{\alpha^2 (kR)^2}{1/2 + (kR)^2 + \sqrt{1/4 + (kR)^2}}
  \,.
\]
The latter inequality implies
\begin{equation}\label{eq:ineq}
	\sqrt{\frac{1}{4R^4} + \frac{k^2}{R^2}} > \alpha^2 - \frac{1}{2R^2} - k^2
  \,.
\end{equation}	
If the right-hand side of~\eqref{eq:ineq} is negative, then
\[
	-k^2 < -\alpha^2 + \frac{1}{2R^2} 
	= -\alpha^2 + \frac{1}{R}\frac{1}{2R}
	\le -\alpha^2 - \frac{\alpha}{2R} < -\alpha^2 - \frac{\alpha}{R}
  \,,
\]
which is the desired inequality.
If the right-hand side of~\eqref{eq:ineq} is non-negative, 
we take the squares of both the right- and left-hand sides
of~\eqref{eq:ineq} and obtain
\[
	\frac{1}{4R^4} + \frac{k^2}{R^2} > 
	\left(\alpha^2 - \frac{1}{2R^2} - k^2\right)^2
  =    \left(\alpha^2 - \frac{1}{2R^2}\right)^2 
	- 2k^2\left(\alpha^2 - \frac{1}{2R^2}\right)+ k^4
  \,.
\]
This inequality is equivalent to
\[
	0 > 
	(\alpha^2 - k^2)^2 - \frac{\alpha^2}{R^2}
  \,.
\]
Consequently,
\[
	\alpha^2 - k^2 < -\frac{\alpha}{R}
  \,,
\]
which again yields the desired upper bound.
\end{proof} 

We notice that analogous upper and lower bounds for $\lambda_1^\alpha(B_R)$
have been recently established in~\cite[Thm.~3]{AFK}.
Moreover, it has been shown in~\cite[Thm.~5]{AFK} that 
$R \mapsto \lambda_1^\alpha(B_R)$ is strictly increasing.
Now we have a reversed monotonicity result.
\begin{Proposition}\label{Prop.monotonicity}
If~$\alpha$ is negative, then 
$
  R \mapsto \lambda_1^\alpha(B_R^\mathrm{ext})
$
is strictly decreasing.
\end{Proposition}
\begin{proof}
We follow the strategy of the proof of~\cite[Thm.~5]{AFK}.
Computing the derivative of $\lambda_R := \lambda_1^\alpha(B_{R}^\mathrm{ext})$
using the differential equation that~$\lambda_R$ satisfies, 
one finds (\cf~\cite[Lem.~2]{AFK})
\begin{equation}\label{derivative}
  \frac{\partial \lambda_R}{\partial R} 
  = - \frac{2}{R} \, \lambda_R 
  + \alpha \, \frac{ \psi_R(R)^2 }
  {\displaystyle \int_R^\infty \psi_R(r)^2 \, r \, \der r}
  \,,
\end{equation}
where $\psi_R(r) := K_0(k r)$ is the eigenfunction 
corresponding to~$\lambda_R = -k^2$.
Employing the formula
\begin{equation}\label{identity}
  \int_R^\infty K_0(kr)^2 \, r \, \der r
  = \frac{r^2}{2}\left[
  	K_0(k r)^2 - K_1(k r)^2
  	\right]\bigg|_{r=R}^{r= \infty} =  	
  \frac{R^2}{2} \left[
  K_1(k R)^2 - K_0(k R)^2
  \right]
\end{equation}
and~\eqref{implicit}, we eventually arrive at 
the equivalent identity for the eigenvalue derivative
\begin{equation}\label{eq:lambda_R_derivative} 
\begin{split}
  \frac{\partial \lambda_R}{\partial R} 
	= - \frac{2}{R} \, \lambda_R \,
  \frac{\displaystyle \lambda_R+\alpha^2+\frac{\alpha}{R}}{\lambda_R+\alpha^2}
  \,.
\end{split}  
\end{equation}
The proof is concluded by recalling Proposition~\ref{Prop.bounds},
which implies that the right-hand side is negative.
\end{proof} 
%

\section{Proof of Theorem~\ref{Thm}}\label{Sec.proof}
%
Now we are in a position to establish Theorem~\ref{Thm}.
Throughout this section, 
$\Omega\subset\Real^2$~is a \emph{convex} bounded open set
with smooth boundary $\Sigma :=\partial\Omega$;
then $\Omega^\mathrm{ext}$ is necessarily connected.
We also assume that~$\alpha$ is negative.

The main idea of the proof is to parameterise $\Omega^\mathrm{ext}$
by means of the \emph{parallel coordinates}
\begin{equation}\label{tube} 
  \mathcal{L} : \Sigma \times (0,\infty) \to \Omega^\mathrm{ext} :
  \left\{(s,t) \mapsto s + n(s) \, t \right\}
  \,,
\end{equation}
where~$n$ is the outer unit normal to~$\Omega$ as above. 
Notice that~$\mathcal{L}$ is indeed a diffeomorphism
because of the convexity and smoothness assumptions. 
To be more specific, the metric induced by~\eqref{tube} 
acquires the diagonal form
\begin{equation}\label{metric} 
  \der\mathcal{L}^2 = (1-\kappa(s)\,t)^2 \, \der s^2 + \der t^2
  \,,
\end{equation}
where $\kappa := -\der n$ is the \emph{curvature} of~$\partial\Omega$.
By our choice of~$n$, 
the function~$\kappa$ is non-positive because~$\Omega$ is convex 
(\cf~\cite[Thm.~2.3.2]{Kli}). 
Consequently, the Jacobian of~\eqref{tube} given by
$
  1-\kappa(s)\,t
$
is greater than or equal to~$1$.
In particular, it is positive and therefore~$\mathcal{L}$
is a local diffeomorphism by the inverse function theorem
(see, \eg, \cite[Thm. 0.5.1]{Kli}).
To see that it is a global diffeomorphism,
notice that~~$\mathcal{L}$ is injective, 
because of the convexity assumption,
and that~~$\mathcal{L}$ is surjective
thanks to the smoothness of $\Omega$.

Summing up, $\Omega^\mathrm{ext}$ can be identified 
with the product manifold $\Sigma \times (0,\infty)$
equipped with the metric~\eqref{metric}.
Consequently, the Hilbert space $\sii(\Omega^\mathrm{ext})$ 
can be identified with 
$$
  \mathcal{H} :=
  \sii\big(\Sigma \times (0,\infty),(1-\kappa(s)\,t) \, \der s \, \der t\big)
  \,.
$$ 
The identification is provided by the unitary transform
\begin{equation}\label{eq:U}
 U : \sii(\Omega^\mathrm{ext}) \to \mathcal{H} :
 \{u \mapsto u \circ \mathcal{L} \}
 \,.
\end{equation}
It is thus natural to introduce the unitarily equivalent operator
$
  H_\alpha := U (-\Delta_\alpha^{\Omega^\mathrm{ext}}) U^{-1}
$,
which is the operator associated with the transformed form
$
  h_\alpha[\psi] :=  Q_\alpha^{\Omega^\mathrm{ext}}[U^{-1}\psi]
$, 
$\Dom(h_\alpha) := U \Dom(Q_\alpha^{\Omega^\mathrm{ext}})$.
Of course, we have the equivalent characterisation 
of the lowest eigenvalue
\begin{equation}\label{Rayleigh.parallel}
  \lambda_1^\alpha(\Omega^\mathrm{ext}) 
  = \inf_{\stackrel[\psi\not=0]{}{\psi \in \Dom(h_\alpha)}} 
  \frac{\displaystyle h_\alpha[\psi]}
  {\displaystyle \, \|\psi\|_\mathcal{H}^2 }
  \,.
\end{equation}

The set of restrictions of functions from $C_0^\infty(\Real^2)$ to $\Omega^\mathrm{ext}$
is a core of $Q_\alpha^{\Omega^\mathrm{ext}}$.
Taking~$u$ from this core, it is easily seen that 
$\psi := U u$ is a restriction of a function $C_0^\infty(\partial\Omega \times \Real)$
to $\partial\Omega \times (0,\infty)$ and that
$$
\begin{aligned}
  h_\alpha[\psi] 
  &= \int_{\Sigma\times(0,\infty)} 
  \left(\frac{\displaystyle
  |\partial_s\psi(s,t)|^2}{1-\kappa(s) \, t} + 
  |\partial_t\psi(s,t)|^2 \, (1-\kappa(s) \, t)\right)  \, \der s \, \der t
  + \alpha \int_{\Sigma} |\psi(s,0)|^2 \, \der s
  \,,
  \\
  \|\psi\|_\mathcal{H}^2 
  &= \int_{\Sigma\times(0,\infty)} 
  |\psi(s,t)|^2 \, (1-\kappa(s) \, t) \, \der s \, \der t
  \,.
\end{aligned}
$$
Restricting in~\eqref{Rayleigh.parallel} to test functions 
with level lines parallel to~$\Sigma$,
\ie\ taking~$\psi$ independent of~$s$,
we obtain 
\begin{equation}\label{Rayleigh.bound}
  \lambda_1^\alpha(\Omega^\mathrm{ext}) 
  \leq
  \inf_{\stackrel[\psi\not=0]{}{\psi \in C_0^\infty([0,\infty))}} 
  \frac{\displaystyle
  \int_0^\infty 
  |\psi'(t)|^2 \, (|\Sigma|+2\pi t) \, \der t
  + \alpha \, |\Sigma| \, |\psi(0)|^2}
  {\displaystyle
  \int_0^\infty
  |\psi(t)|^2 \, (|\Sigma|+2\pi t) \, \der t} 
  \,.
\end{equation}
Here we have used the geometric identity $\int_{\Sigma} \kappa = -2\pi$
(see, \eg, \cite[Cor.~2.2.2]{Kli}).

Now, assume that the perimeter is fixed, 
\ie\ $|\Sigma| = c_1$.
Since the perimeter is the only geometric quantity 
on which the right-hand side of~\eqref{Rayleigh.bound} depends
and since the eigenfunction corresponding to 
$\lambda_1^\alpha(B_{R_1}^\mathrm{ext})$ is radially symmetric
(therefore independent of~$s$ in the parallel coordinates),
we immediately obtain 
\begin{equation}\label{immediate}
  \lambda_1^\alpha(\Omega^\mathrm{ext}) \leq \lambda_1^\alpha(B_{R_1}^\mathrm{ext})
\end{equation}
for any~$\Omega$ of the fixed perimeter~$c_1$. 
This proves the isoperimetric optimisation result of Theorem~\ref{Thm}.
 
To establish the isochoric optimisation result of Theorem~\ref{Thm},
we notice that since $|\Sigma| = |\partial B_{R_1}|$,
the classical (geometric) isoperimetric inequality implies $|\Omega| \leq |B_{R_1}|$,
with equality if and only if~$\Omega$ is the disk. 
Hence, there exists $B_{R_2} \subset B_{R_1}$ such that $|B_{R_2}| = |\Omega|$.
By~\eqref{immediate}, it is then enough to show that 
$\lambda_1^\alpha(B_{R_1}^\mathrm{ext}) \leq \lambda_1^\alpha(B_{R_2}^\mathrm{ext})$.
This inequality (and therefore the second result of Theorem~\ref{Thm})
is a consequence of the more general monotonicity
result of Proposition~\ref{Prop.monotonicity}.
\hfill \qed

\section{Conclusions}\label{Sec.end}
%
Let us conclude the paper by several comments on our results.

\subsection{Necessity of convexity}\label{Sec.counterexample}
The assumption on $\Omega$ to be convex is necessary
in view of the following simple counterexample. Let $\Omega\subset\mathbb{R}^2$ 
be the union of two disks $B_{R_3}(x_1)$ and $B_{R_3}(x_2)$ 
of the same radius $R_3 > 0$ whose 
centres~$x_1$ and~$x_2$ are chosen in such a way that the closures of the disks in $\mathbb{R}^2$ are disjoint.
According to~\cite[Thm. 1.1]{KP} (see also \cite[Thm.~3]{FK7}), we have
\begin{align*}
	\lambda^\alpha_1(\Omega^{\rm ext}) & = 
	-\alpha^2 - \frac{\alpha}{R_3} + o(\alpha) \,, &
	 \alpha\rightarrow -\infty \,, \\
	\lambda^\alpha_1(B_R^{\rm ext})&  = -\alpha^2 - \frac{\alpha}{R} + o(\alpha) \,,
	 &\alpha\rightarrow -\infty \,.
\end{align*}
The constraints $|\partial\Omega| = |\partial B_{R_1}|$ and $|\Omega| = |B_{R_2}|$ yield that $R_1 = 2R_3$ and $R_2 = \sqrt{2}R_3$, respectively. 
Taking into account the above large coupling
asymptotics and the relations between the radii,
we observe that for $\alpha < 0$ with sufficiently large $|\alpha|$
the reverse inequalities 
$\lambda^\alpha_1(B_{R_1}^{\rm ext}) < \lambda^\alpha_1(\Omega^{\rm ext})$
and 
$\lambda^\alpha_1(B_{R_2}^{\rm ext}) < \lambda^\alpha_1(\Omega^{\rm ext})$ 
are satisfied.

We point out that, while the domain~$\Omega$ 
of the above counterexample is disconnected,
its exterior~$\Omega^\mathrm{ext}$ is still connected.
We leave it as an open question whether there exists a counterexample
in the class of connected non-convex domains~$\Omega$.

\subsection{Uniqueness of the optimiser}\label{Sec.uniqueness}
In this subsection, 
we demonstrate that the exterior of the disk is the \emph{unique} maximiser
for both the isochoric and isoperimetric 
spectral optimisation problems of Theorem~\ref{Thm}. 
\begin{Theorem}
Let~$\alpha$ be negative. 
For all convex smooth bounded open sets~$\Omega\subset\Real^2$ 
of a fixed perimeter  
(respectively, fixed area)
different from a disk~$B_{R_1}$ of the same perimeter
(respectively, from a disk $B_{R_2}$ of the same area), 
we have a strict inequality
\begin{flalign}\label{unique}
  \fbox{$\alpha < 0$} &&
  \lambda_1^\alpha(\Omega^\mathrm{ext}) 
  < \lambda_1^\alpha(B_{R_1}^\mathrm{ext})
  \qquad 
  \mbox{(respectively,} \quad 
  \lambda_1^\alpha(\Omega^\mathrm{ext}) 
  < \lambda_1^\alpha(B_{R_2}^\mathrm{ext})
  \mbox{)}
  \,.
  &&
\end{flalign}
\end{Theorem}
\begin{proof}
As usual, $\lambda_1^\alpha := \lambda_1^\alpha(\Omega^\mathrm{ext})$ and $u_1^\alpha$ 
denote respectively the lowest eigenvalue and the corresponding eigenfunction
of $-\Delta^{\Omega^\mathrm{ext}}_\alpha$ with $\alpha < 0$.
Without loss of generality, we assume that $u_1^\alpha$
is positive everywhere in $\Omega^\mathrm{ext}$.
Furthermore, we introduce the auxiliary function $\psi := U u_1^\alpha$
where the unitary transform $U$ is as in~\eqref{eq:U}.
In view of Theorem~\ref{Thm} and inequality~\eqref{Rayleigh.bound}
in its proof, non-uniqueness of the optimiser
for the spectral isoperimetric problem would necessarily imply
the existence of a non-circular domain $\Omega$
for which the function~$\psi$ is independent of~$s$; 
\ie~its level lines are parallel to $\Sigma := \partial\Omega$. 
In the sequel, with a slight abuse of notation, 
we use the same symbol~$\psi$ to denote
the function $t\mapsto \psi(t)$ of a single variable.

Restricting to test functions with support lying inside~$\Omega^\mathrm{ext}$,
the variational characterisation of~$\psi$ implies
\[
	\forall \varphi \in C^\infty_0(\Sigma \times (0,\infty)) 
  \,, \qquad 
  h_\alpha(\varphi,\psi) = \lambda_1^\alpha (\varphi,\psi)_\cH
  \,.
\]
It is enough to consider real-valued test functions~$\varphi$ only.
Taking into account that~$\psi$ is independent of~$s$,
we end up with the identity
\[
	\int_{\Sigma\times (0,\infty)}
	\partial_t \varphi(s,t) \, \psi'(t) \, (1-\kappa(s)t)
        \, \der s \, \der t = \lambda_1^\alpha 
	\int_{\Sigma\times (0,\infty)}
	\varphi(s,t) \, \psi(t) \, (1-\kappa(s)t)
	\, \der s \, \der t
\]
valid for all real-valued $\varphi \in C^\infty_0(\Sigma \times (0,\infty))$.
Now we restrict our attention to test functions 
of the type $\varphi(s,t) = \varphi_1(s)\varphi_2(t) 
\in C^\infty_0(\Omega^\mathrm{ext})$ with
$\varphi_1 \in C^\infty(\Sigma)$ and
$\varphi_2 \in C^\infty_0((0,\infty))$.
Then the above displayed equation reduces to
\[
	\int_\Sigma\varphi_1(s)\int_0^\infty 
        \varphi_2'(t) \, \psi'(t)\,  
	(1-\kappa(s)t) \, \der t \, \der s 
	=
	\lambda_1^\alpha \int_\Sigma
	\varphi_1(s)\int_0^\infty \varphi_2(t) \, \psi(t) \, (1-\kappa(s)t) 
        \, \der t \, \der s.
\]
Density of $C^\infty(\Sigma)$ in $L^1(\Sigma)$ gives us
\begin{equation}\label{eq:last}
	\forall s\in\Sigma \,, \ 
	\varphi_2 \in C^\infty_0((0,\infty)) \,,
	\quad
	\int_0^\infty 
        \varphi_2^\prime(t) \, \psi^\prime(t) \,  
	(1-\kappa(s)t) \, \der t 
	=
	\lambda_1^\alpha \int_0^\infty 
        \varphi_2(t) \, \psi(t) \, (1-\kappa(s)t) \, \der t.
\end{equation}
Since $\Omega$ is not a disk, 
there exist $s_1, s_2 \in \Sigma$ such that $\kappa(s_1) \ne \kappa(s_2)$ 
(see, \eg, \cite[Prop.~1.4.3]{Kli}).
Taking the difference of~\eqref{eq:last} with $s = s_1$ and with $s = s_2$, 
we eventually get
\begin{equation}\label{eq:finalcondition}
		\forall \varphi_2 \in C^\infty_0((0,\infty)) \,,
		\qquad
		\int_0^\infty 
 \varphi_2^\prime(t) \, \psi^\prime(t)  
		\, t \, \der t 
		=
		\lambda_1^\alpha 
		\int_0^\infty 
 \varphi_2(t) \, \psi(t) \, t \,\der t \,.
\end{equation}
Let us fix a function $\eta \in C^\infty_0((0,\infty))$ which 
satisfies the following properties:
\begin{itemize}\setlength{\parskip}{0.6ex}
	\item [(i)] $0\le \eta \le 1$,
	\item [(ii)] $\eta(t) = 1$ for all $t \in [1,2]$,
	\item [(iii)] $\supp \eta \subset [0,3]$.
\end{itemize}
%
Furthermore, 
for every positive integer~$n$,
we define a function
$\eta_n\in C^\infty_0((0,\infty))$ by
\begin{equation}
	\eta_n(t) := \begin{cases}
					\eta(nt),	& t\in \left(0,\frac{2}{n}\right),\\
					1,			& t\in \left(\frac{2}{n},n+1\right),\\
					\eta(t-n),& t\in (n+1,\infty) \,. 
				\end{cases}	
\end{equation}
Now we plug $\varphi_2 = \eta_n\psi\in C^\infty_0((0,\infty))$ into~\eqref{eq:finalcondition}.
By the dominated convergence theorem
(using that $t\mapsto |\psi(t)|^2 \, t$ is integrable), we obtain
\begin{equation}\label{eq:limRHS}
		\int_0^\infty |\psi(t)|^2 \, \eta_n(t) \, t \, \der t
		\xrightarrow[n \to \infty]{} 
\int_0^\infty |\psi(t)|^2 \, t \, \der t
\,.
\end{equation}
The left-hand side in~\eqref{eq:finalcondition} with $\varphi_2 = \eta_n\psi$
can be rewritten as
\begin{equation}\label{eq:decomposition}
	I_n := \int_0^\infty 
   (\eta_n\psi)^\prime(t) \, \psi^\prime(t) \, t \, \der t 
	=
	\int_0^\infty 
	|\psi^\prime(t)|^2 \, \eta_n(t) \, t \, \der t
	+
	\int_0^\infty 
	\psi^\prime(t) \, \psi(t) \, \eta_n^\prime(t) \, t \, \der t \,. 
\end{equation}
For the first term on the right-hand side in~\eqref{eq:decomposition}
we get  
\begin{equation}\label{eq:Jn}
	I_n^{(1)}  := 
	\int_0^\infty 
	|\psi^\prime(t)|^2 \, \eta_n(t) \, t \, \der t
	\xrightarrow[n \to \infty]{} 
\int_0^\infty |\psi^\prime(t)|^2 \, t \, \der t
\,,
\end{equation}
by the dominated convergence (using that $t\mapsto |\psi^\prime(t)|^2 \, t$ is integrable).
The second term on the right-hand side in~\eqref{eq:decomposition}
can be further transformed as
\begin{equation}\label{eq:decomposition2}
\begin{split}
	\int_0^\infty 
	\psi^\prime(t) \, \psi(t) \, \eta_n^\prime(t) \, t \, \der t
	& = 	
	n\int_0^{2/n} \psi^\prime(t) \, \psi(t) \, \eta^\prime(nt) \, t \, \der t
	+
	\int_{n+1}^{\infty} \psi^\prime(t) \, \psi(t) \, \eta^\prime(t-n) \, t \, \der t\\
	& = 
	\int_0^{2}  \psi^\prime\left(\frac{r}{n}\right) \psi\left(\frac{r}{n}\right)
	\frac{\eta^\prime(r) r}{n} \, \der r
	+
	\int_1^3 
	\psi^\prime(r\!+n) \, \psi(r\!+n) \,
	\eta^\prime(r) \, (r\!+n) \, \der r.
\end{split}	
\end{equation}
Again making use of the dominated convergence theorem,
we obtain
\begin{equation}\label{eq:Kn}
		I_n^{(2)} := 
		\int_0^{2} \psi^\prime\left(\frac{r}{n}\right) \psi\left(\frac{r}{n}\right)
		\frac{\eta^\prime(r) \, r}{n} \, \der r
		\xrightarrow[n \to \infty]{} 0 
\,;
\end{equation}		
here we implicitly employed that the integrand is uniformly bounded in $n\in\dN$.
Observe that
\[
	\left|\sum_{n=1}^\infty
	\int_1^3 
	\psi^\prime(r+n) \, \psi(r+n) \,
	\eta^\prime(r) \, (r+n) \, \der r \right| 
	\le 
	2 \, \|\eta^\prime\|_\infty
	\int_0^\infty
	|\psi(t)\,\psi^\prime(t)| \, t \, \der t < \infty
 \,,
\]
where finiteness of the latter integral follows from the fact that
$\psi \in \Dom(h_\alpha)$. Therefore, we infer
\begin{equation}\label{eq:Ln}
	I_n^{(3)} := 	\int_1^3 
	\psi^\prime(r+n) \, \psi(r+n) \,
	\eta^\prime(r) \, (r+n) \, \der r 
        \xrightarrow[n\to\infty]{} 
        0 \,.
\end{equation}
Combining the decompositions~\eqref{eq:decomposition},~\eqref{eq:decomposition2}
with the limits~\eqref{eq:Jn},~\eqref{eq:Kn},~\eqref{eq:Ln}, 
we arrive at
\begin{equation}\label{eq:limLHS}
	I_n 
		=
	I_n^{(1)}
	+
	I_n^{(2)}
	+
	I_n^{(3)}
	\xrightarrow[n\to\infty]{} 
       \int_0^\infty |\psi^\prime(t)|^2 \, t \, \der t
\,. 
\end{equation}
The limits~\eqref{eq:limRHS},~\eqref{eq:limLHS} 
and the condition~\eqref{eq:finalcondition} imply
\[
	\int_0^\infty 
	| \psi^\prime(t)|^2
	\, t \, \der t 
	=
	\lambda_1^\alpha 
	\int_0^\infty |\psi(t)|^2 \, t \, \der t \,.
\]
Finally, taking into account that $\lambda_1^\alpha$ is negative,
we get a contradiction.
This completes the proof of the first strict inequality in~\eqref{unique}.

To show that disk is the unique optimiser for the spectral isochoric
inequality is much simpler than in the isoperimetric case. 
Suppose that there exists a non-circular domain $\Omega$ for which
$\lambda_1^\alpha(\Omega^\mathrm{ext}) = \lambda_1^\alpha(B^\mathrm{ext}_{R_{2}})$
with $|\Omega| = |B_{R_2}|$.
Note that for $B_{R_1}$ with $|\partial B_{R_1}|=  |\partial\Omega|$ we get
$R_2 < R_1$ using the standard geometric isoperimetric inequality.
Thus, Theorem~\ref{Thm} and 
Proposition~\ref{Prop.monotonicity} imply
\[
	\lambda_1^\alpha(\Omega^\mathrm{ext}) = \lambda_1^\alpha(B^\mathrm{ext}_{R_{2}})
	> 
	\lambda_1^\alpha(B^\mathrm{ext}_{R_1}) 
	\ge
	\lambda_1^\alpha(\Omega^\mathrm{ext})
  \,,  
\]	
which is obviously a contradiction.
\end{proof}
\begin{remark}
	As a consequence of the first claim in Theorem~\ref{Thm}, we obtain
	the following quantitative improvement upon the second inequality of~\eqref{unique}
	\begin{equation}\label{stability}
		\lambda_1^\alpha(\Omega^\mathrm{ext}) 
		\le 
		\lambda_1^\alpha(B_{R_2}^\mathrm{ext}) - 
		\big[\lambda_1^\alpha(B_{R_2}^\mathrm{ext}) - \lambda_1^\alpha(B_{R_1}^\mathrm{ext})\big]
\,,
	\end{equation}
	where the radii~$R_1$ and~$R_2$ can be easily 
	expressed through $|\partial\Omega|$ and $|\Omega|$, respectively,
	by virtue of the relations 
	$|B_{R_1}| = |\partial\Omega|$ and $|B_{R_2}| = |\Omega|$.
	In view of the inequality $R_2 < R_1$, 
	the difference $\lambda_1^\alpha(B_{R_2}^\mathrm{ext}) - \lambda_1^\alpha(B_{R_1}^\mathrm{ext})$ is positive by Proposition~\ref{Prop.monotonicity},
	so~\eqref{stability} indeed represents a quantified version of 
	the reverse spectral isochoric inequality in the spirit of~\cite{BP12, BuFeNiTr}. 
	More careful analysis of the derivative 
	in~\eqref{eq:lambda_R_derivative} can be used to get a positive
	lower bound on this difference in terms of $R_1$ and $R_2$.
\end{remark}

\subsection{Higher dimensions}\label{Sec.high}
We have already noticed that $\lambda_1^\alpha(\Omega^\mathrm{ext})$ 
is not necessarily a discrete eigenvalue in higher dimensions.
For any dimension $d \ge 3$, however, there exists a critical value $\alpha_0 < 0$
depending on~$\Omega$ 
such that $\lambda_1^\alpha(\Omega^\mathrm{ext})$ is a discrete eigenvalue
if, and only if, $\alpha < \alpha_0$,
so the optimisation problem in the exterior
of a compact set becomes non-trivial in this regime. 
In this subsection, we argue that no analogue of Theorem~\ref{Thm}
can be expected if $d \ge 3$.

To this aim we construct a simple counterexample
	which relies on the large coupling asymptotics for the lowest eigenvalue.
	First, we fix a ball $B_R\subset\Real^d$ of arbitrary radius $R > 0$.
	Further, let $\Omega_0\subset\Real^d$ be the union of two disjoint balls 
$B_r(x_1)$ and $B_r(x_2)$ of the same sufficiently small radius $r > 0$ whose centers~$x_1$ and~$x_2$ are located at a distance $L > 0$. 
Finally, we define the
	domain $\Omega$ as the convex hull of $\Omega_0$. 
By choosing $L > 0$ large enough, 
we can satisfy either of the constraints
$|\partial \Omega| = |\partial B_R|$ or $|\Omega| = |B_R|$. 
It can be easily checked
	that the domain $\Omega$ has a $C^{1,1}$ boundary 
and that the mean curvature
	of $\partial\Omega$ is piecewise constant, being equal to $-1/r$
	on the hemispheric cups and to $-\frac{d-2}{(d-1)r}$ on the cylindrical face
(in agreement with the rest of this paper,
we compute the mean curvature with respect to the outer normal 
to the bounded set~$\Omega$).
Applying~\cite[Thm. 1.1]{KP}, we arrive at
	\begin{align*}
		\lambda_1^\alpha(\Omega^\mathrm{ext}) 
		& = 
		-\alpha^2 - \frac{\alpha\,(d-2)}{r} + o(\alpha)\,,&\alpha\rightarrow-\infty\,,\\
		\lambda_1^\alpha(B_R^\mathrm{ext}) 
		& = 
		-\alpha^2 - \frac{\alpha(d-1)}{R} + o(\alpha)\,,&\alpha\rightarrow-\infty\,.
	\end{align*}	
	In view of the above asymptotics, we infer that 
	for $r < \frac{d-2}{d-1} R$ and for $\alpha < 0$ 
	with sufficiently large $|\alpha|$ the reverse inequality 
	$\lambda^\alpha_1(B_{R}^{\rm ext}) < \lambda^\alpha_1(\Omega^{\rm ext})$
	holds.
	
	We expect that a counterexample based on a ($C^\infty$-)smooth 
domain can also be constructed with additional technical efforts.
	
\subsection{More on dimension three}\label{Sec.three}
The previous subsection demonstrates that,
contrary to the two-dimensional situation, 
the exterior of the ball can be a global maximiser 
neither for the isoperimetric nor isochoric problems.
Let us look at where the technical approach of the present paper
fails in dimension three.

Let~$\Omega$ be a convex smooth bounded open set in~$\Real^3$.
In this case, the usage of parallel coordinates based 
on~$\Sigma:=\partial\Omega$ 
and restricting to test functions depending only on the distance 
to the boundary yield
\begin{equation}\label{Rayleigh.parallel.3D}
  \lambda_1^\alpha(\Omega^\mathrm{ext}) 
  \leq \inf_{\stackrel[\psi\not=0]{}{\psi \in C_0^\infty([0,\infty))}} 
  \frac{\displaystyle
  \int_{\Sigma\times(0,\infty)} 
  |\psi'(t)|^2 \, (1-2M(s) \, t+K(s) \, t^2) \, \der \Sigma \, \der t
  + \alpha \int_{\Sigma} |\psi(0)|^2 \, \der \Sigma}
  {\displaystyle
  \int_{\Sigma\times(0,\infty)} 
  |\psi(t)|^2 \, (1-2M(s) \, t+K(s) \, t^2) \, \der \Sigma \, \der t }
  \,.
\end{equation}
Here~$\der\Sigma:=|g|^{1/2}(s) \, \der s$ is the surface measure of~$\Sigma$,
with~$g$ being the Riemannian metric of~$\Sigma$ induced by the embedding
of~$\Sigma$ in~$\Real^3$, 
and $K$ and~$M$ denote respectively 
the \emph{Gauss curvature} and the \emph{mean curvature} of~$\Sigma$
(see~\cite{CEK} for more geometric details). 
$K$~is an intrinsic quantity, while~$M$ is non-positive 
when computed with respect to our choice (outer to~$\Omega$)
of the normal vector field~$n$.

By definition, $\int_{\Sigma} 1 \, \der\Sigma$ 
equals the total area $|\Sigma|$ of~$\Sigma$,
while $\int_{\Sigma} K(s) \, \der\Sigma = 4\pi$
by the Gauss-Bonnet theorem for closed surfaces
diffeomorphic to the sphere (see~\cite[Thm.~6.3.5]{Kli}).
The quantity $\mathcal{M}_\Sigma:=\int_{\Sigma} |M(s)| \, \der\Sigma$ 
is known as the half of the \emph{total mean curvature} of~$\Sigma$  
(see~\cite[\S~28.1.3]{BuZa}).
Moreover, we have \cite[\S~19]{BuZa}
$
  \mathcal{M}_\Sigma = 2 \pi \, b(\Omega)
$,
where $b(\Omega)$ is the \emph{mean width} of~$\Omega$. 
Consequently,
\begin{equation} 
  \lambda_1^\alpha(\Omega^\mathrm{ext}) 
  \leq \inf_{\stackrel[\psi\not=0]{}{\psi \in C_0^\infty([0,\infty))}} 
  \frac{\displaystyle
  \int_0^\infty
  |\psi'(t)|^2 \, (|\Sigma|+2 \, \mathcal{M}_\Sigma \, t + 4\pi \, t^2) \, 
  \der t
  + \alpha \, |\Sigma| \, |\psi(0)|^2}
  {\displaystyle
  \int_0^\infty
  |\psi(t)|^2 \, (|\Sigma|+2 \, \mathcal{M}_\Sigma \, t+ 4\pi \, t^2) \, 
   \der t }
  \,.
\end{equation}

To get a reverse spectral isoperimetric inequality as in the planar case above, 
we would need in addition to the constraint $|\partial\Omega|=c_1$
also require that the mean width $b(\Omega)$ is fixed. 
However, the Minkowski quadratic inequality 
for cross-sectional measures (\cf~\cite[\S~20.2]{BuZa}), 
$
  \mathcal{M}_\Sigma^2 \geq 4\pi \, |\Sigma| 
$,
with equality only if~$\Omega$ is a ball, 
implies that the two simultaneous constraints 
are possible only if either the class of admissible domains excludes the ball
or the class of admissible domains consists of the ball only.
In the first case our method is not applicable, while in the second case
	the method can be applied but it yields a trivial statement.

\subsection*{Acknowledgments}
D.K.\ was partially supported by FCT (Portugal)
through project PTDC/\-MAT-CAL/4334/2014.
V.L.\ acknowledges the support by the grant
No.~17-01706S of the Czech Science Foundation (GA\v{C}R).

%

\begin{thebibliography}{10}
	
	\bibitem{Abramowitz-Stegun}
	M.~S. Abramowitz and I.~A. Stegun{, eds.}, \emph{Handbook of mathematical
		functions}, Dover, New York, 1965.
	
	\bibitem{AFK}
	P.~R.~S. Antunes, P.~Freitas, and D.~Krej\v{c}i\v{r}\'{\i}k, \emph{Bounds and
		extremal domains for {R}obin eigenvalues with negative boundary parameter},
	Adv. Calc. Var., to appear; preprint on arXiv:1605.08161 [math.SP];
	doi:10.1515/acv-2015-0045.
	
	
	\bibitem{Bareket_1977}
	M.~Bareket, \emph{On an isoperimetric inequality for the first eigenvalue of a
		boundary value problem}, SIAM J. Math. Anal. \textbf{8} (1977), 280--287.
	
	\bibitem{BLLR15}
	J.~Behrndt, M.~Langer, V.~Lotoreichik, and J.~Rohleder,
	\emph{Quasi boundary triples and semibounded self-adjoint extensions},
	Proc. Roy. Soc. Edinburgh Sect. A., to appear; preprint on arXiv:1504.03885 [math.SP];
	doi:10.1017/S0308210516000421.  
	

	\bibitem{Blasjo_2005}
	V.~Bl{\aa}sj{\"o}, \emph{The isoperimetric problem}, Am. Math. Mon.
	\textbf{112} (2005), 526--566.
	
	\bibitem{Bossel_1986}
	M.-H. Bossel, \emph{Membranes {\'{e}}lastiquement li{\'{e}}es: {E}xtension du
		th{\'{e}}or{\'{e}}me de {R}ayleigh-{F}aber-{K}rahn et de
		l'in{\'{e}}galit{\'{e}} de {C}heeger}, C. R. Acad. Sci. Paris S\'{e}r. I
	Math. \textbf{302} (1986), 47--50.
	
	\bibitem{BP12}
	L.~Brasco and A.~Pratelli, 
	\emph{Sharp stability of some spectral inequalities},
	Geom. Funct. Anal. \textbf{22} (2012), 107--135.
	
	\bibitem{Brock-Daners_2007}
	F.~Brock and D.~Daners, \emph{Conjecture concerning a {F}aber-{K}rahn
		inequality for {R}obin problems}, Oberwolfach Rep. \textbf{4} (2007),
	1022--1023, Open Problem in Mini-Workshop: Shape Analysis for Eigenvalues
	(Organized by: D. Bucur, G. Buttazzo, A. Henrot).
	
	\bibitem{BuFeNiTr}
	D.~Bucur, V.~Ferone, C.~Nitsch, and C.~Trombetti,
	The quantitative Faber-Krahn inequality for the Robin Laplacian,
	preprint on arXiv:1611.06704 [math.AP].	
	
	\bibitem{BuZa}
	Yu.~D. Burago and V.~A. Zalgaller, \emph{Geometric inequalities},
	Springer-Verlag, Berlin Heidelberg, 1988.
	
	\bibitem{CEK}
	G.~Carron, P.~Exner, and D.~Krej\v{c}i\v{r}\'{\i}k, \emph{Topologically
		nontrivial quantum layers}, J.~Math.\ Phys. \textbf{45} (2004), 774--784.
	
	\bibitem{Daners_2006}
	D.~Daners, \emph{A {F}aber-{K}rahn inequality for {R}obin problems in any space
		dimension}, Math. Ann. \textbf{335} (2006), 767--785.
	
	\bibitem{Daners_2013}
	\bysame, \emph{Principal eigenvalues for generalised indefinite {R}obin
		problems}, Pot. Anal. \textbf{38} (2013), 1047--1069.
	
	\bibitem{Faber_1923}
	G.~Faber, \emph{Beweis dass unter allen homogenen {M}embranen von gleicher
		{F}l{\"a}che und gleicher {S}pannung die kreisf{\"o}rmige den tiefsten
		{G}rundton gibt}, Sitz. bayer. Akad. Wiss. (1923), 169--172.
	
	\bibitem{Ferone-Nitsch-Trombetti_2015}
	V.~Ferone, C.~Nitsch, and C.~Trombetti, \emph{On a conjectured reversed
		{F}aber-{K}rahn inequality for a {S}teklov-type {L}aplacian eigenvalue},
	Commun. Pure Appl. Anal. \textbf{14} (2015), 63--81.
	
	\bibitem{FK7}
	P.~Freitas and D.~Krej\v{c}i\v{r}\'{\i}k, \emph{The first {R}obin eigenvalue
		with negative boundary parameter}, Adv. Math. \textbf{280} (2015), 322--339.
	
	\bibitem{Gilbarg-Trudinger}
	D.~Gilbarg and N.~S. Trudinger, \emph{Elliptic partial differential equations
		of second order}, Springer-Verlag, Berlin, 1983.
	
	\bibitem{Kato}
	T.~Kato, \emph{Perturbation theory for linear operators}, Springer-Verlag,
	Berlin, 1966.
	
	\bibitem{KP}
	H.~Kova\v{r}\'{i}k and K.~Pankrashkin, 
	\emph{On the $p$-Laplacian with Robin boundary conditions and boundary trace theorems},
    Calc. Var. PDE \textbf{56} (2017), 49.
	
	\bibitem{Kli}
	W.~Klingenberg, \emph{A course in differential geometry}, Springer-Verlag, New
	York, 1978.
	
	\bibitem{Krahn_1924}
	E.~Krahn, \emph{{\"U}ber eine von {R}ayleigh formulierte {M}inimaleigenschaft
		des {K}reises}, Math. Ann. \textbf{94} (1924), 97--100.
	
%
%
	\bibitem{Rayleigh_1877}
	J.~W.~S. Rayleigh, \emph{The theory of sound}, Macmillan, London, 1877, 1st
	edition (reprinted: Dover, New York (1945)).
	
	\bibitem{Segura_2011}
	J.~Segura, \emph{Bounds for ratios of modified {B}essel functions and
		associated {T}ur{\'a}n-type inequalities}, J. Math. Anal. Appl. \textbf{374}
	(2011), 516--528.
	
	\bibitem{Weidmann}
	J.~Weidmann, \emph{Linear operators in {Hilbert} spaces}, Springer-Verlag, New
	York Inc., 1980.
	
\end{thebibliography}
\bibliographystyle{amsplain}

\end{document}